\documentclass[10pt,a4paper]{article}

\usepackage[utf8x]{inputenc}
\usepackage{amsfonts}
\usepackage{amsmath, amsthm, amssymb}
\usepackage{verbatim}
\usepackage[stable]{footmisc}
\usepackage[hmargin=2.5cm,vmargin=3cm]{geometry}
\newcommand{\R}{\mathbb{R}}
\newcommand{\J}{\mathcal{J}}
\newcommand{\I}{\mathcal{I}}
\newcommand{\A}{\mathcal{A}}

\newtheorem{thm}{Theorem}[section]
\newtheorem{Lemma}[thm]{Lemma}

\theoremstyle{definition}
\newtheorem{Remark}[thm]{Remark}

\theoremstyle{definition}
\newtheorem{Definition}[thm]{Definition}

\numberwithin{equation}{section}

\usepackage{hyperref}

\title{Stability of  Order Preserving Transforms}
\author{Dan Florentin and Alexander Segal\footnote{This research was partially supported by the Israel Science Foundation...}}
\begin{document}

\date{}
\maketitle

\begin{abstract}
The purpose of this paper is to show stability of order preserving/reversing transforms on the class of non-negative convex functions in $\R^n$, and its subclass, the class of non-negative convex functions attaining $0$ at the origin (these are called "geometric convex functions"). We show that transforms that satisfy conditions which are weaker than order preserving transforms, are essentially close to the order preserving transforms on the mentioned structures.
\end{abstract}

\section{Introduction}
The concept of duality was studied by Artstein-Avidan and Milman in recent papers \cite{AM2010, AM2010Short, AM2009} on dif{}ferent classes which arise from geometric problems. Examples of such classes are the class of convex bodies containing zero, the class of all lower semi-continuous convex functions on $\R^n$, which we denote by $Cvx(\R^n)$, and its subclass - the class of all lower semi-continuous \textit{geometric} convex functions denoted by $Cvx_0(\R^n)$. A convex function $f$ is said to be geometric if it is non-negative and $f(0)=0$.

It turned out that duality on such classes is uniquely defined by simple properties like \textit{order reversion} and \textit{involution} (actually involution is not required and can be replaced by bijectivity). The Legendre transform is an example of such a duality transform that acts on the class of convex functions $Cvx(\R^n)$. When dealing with $Cvx(\R^n)$, it was shown by Artstein-Avidan and Milman \cite{AM2009} that the Legendre transform is essentially the only order reversing transform acting on this class, where "essentially" means up to the choice of scalar product and addition of linear terms.

{\em Note:} The properties of order preservation and involution actually imply preservation of supremum and infimum on the classes. It is also known
that the mentioned classes can be generated with supremum (or infimum) of an extremal family. This concept is not new, and was used by Kutateladze and Rubinov \cite{KutRub}
to discuss Minkowski duality on complete lattices.

Studying the structure of $Cvx_0(\R^n)$ shows that it dif{}fers from $Cvx(\R^n)$. As was shown by Artstein-Avidan and Milman in \cite{AM2010}, there exist essentially two duality (order reversing) bijective transforms - The Legendre transform, and a "geometric duality" transform called $\A$, on the class of geometric convex functions.

Actually, the authors of \cite{AM2010} showed first that there exist essentially two order \textit{preserving} bijections - identity transform $\I$ and the Gauge transform $\J$ which greatly dif{}fers from $\I$. After showing this, using the fact that $\mathcal{L}$ is an involution and the fact that $\J=\mathcal{L}\A = \A\mathcal{L}$, it is easy to see that the order reversing transforms are also uniquely defined. Notice that the results about order reversing transforms are "dual" to the results about order preserving transforms. For details of the mentioned transforms we refer the reader to \cite{AM2010}, and provide the basic definitions for completeness.

\begin{Definition}
The geometric transform $\A:Cvx_0(\R^n)\rightarrow Cvx_0(\R^n)$ is defined as follows:
\[
(\A f)(x) = \begin{cases} \sup_{\{y\in \R^n: f(y) > 0\}} \frac{<x,y>-1}{f(y)} & \text{ if } x  \in \{ y: f(y)=0\}^{\circ} \\ +\infty & \text{ if } x  \not \in \{ y: f(y)=0\}^{\circ} \end{cases}
\]
assuming $\sup \emptyset = 0$.
\end{Definition}
\begin{Definition}
The Legendre transform $\mathcal{L}$ of a function $f$ is defined as follows:
\[
(\mathcal{L} f)(x) = \sup_y(<x,y>-f(y)),
\]
\end{Definition}
and the Gauge transform $\J$ is defined as $\J f=\mathcal{AL}f=\mathcal{LA}f$, for $f \in Cvx_0(\R^n)$. Notice that the commutativity of $\A$ and $\mathcal{L}$ requires a proof, and is actually a non-trivial fact. The Gauge transform $\J$ can be calculated, and written explicitly:
\[
(\J f)(y)=\inf{ \{  1/f(x): y=tx/f(x), 0 \leq t \leq 1\}},
\]
where $\inf\emptyset = +\infty$, and $0/f(0)$ is understood in the sense of limits. \\
In this paper we discuss the stability of the mentioned transforms on the class $Cvx_0(\R^n)$ and $Cvx_{+}(\R^n)$ (non-negative convex functions). We do not deal with classes of convex bodies, and refer the reader to \cite{AMStab} for results on such classes. We start with the following definitions:
\begin{Definition}
Let $\tilde{C} > 1$, and $\tilde{c} = \tilde{C}^{-1}$. A bijective transform $T$, on the class $Cvx_0(\R^n)$ that satisfies the following conditions:
\begin{enumerate}
\item[\ref{TDefinition}a.]{$f \leq g$ implies $Tf \leq \tilde{C}Tg$}
\item[\ref{TDefinition}b.]{$f \leq \tilde{c}g$ implies $Tf \leq Tg$,}
\label{TDefinition}
\end{enumerate}
will be called a $\tilde{C}$-\textit{almost order preserving} transformation, or just almost order preserving, in case there exists some unspecified constant that satisfies conditions (\ref{TDefinition}a) and (\ref{TDefinition}b).
\end{Definition}
Similarly, we can define a $\tilde{C}$-\textit{almost order reversing} transform:
\begin{Definition}
Let $\tilde{C} > 1$, and $\tilde{c} = \tilde{C}^{-1}$. A bijective transform $T$, on the class $Cvx_0(\R^n)$ that satisfies the following conditions:
\begin{enumerate}
\item[\ref{TRevDefinition}a.]{$f \leq g$ implies $Tf > \tilde{c}Tg$}
\item[\ref{TRevDefinition}b.]{$f \leq \tilde{c}g$ implies $Tf > Tg$.}
\end{enumerate}
will be called a $\tilde{C}$-\textit{almost order reversing} transformation, or just almost order reversing, in case there exists some constant that satisfies conditions (a) and (b).
\label{TRevDefinition}
\end{Definition}
\begin{Remark}
If $T$ and $T^{-1}$ are almost order preserving transforms, the following holds:
\begin{enumerate}
\item[\ref{TRevEquiv}a.]{$Tf \leq Tg$ implies $f \leq \tilde{C}g$.\label{stab2}  }
\item[\ref{TRevEquiv}b.]{$Tf \leq \tilde{c}Tg$ implies $f \leq g$.} \label{stab4}
\end{enumerate}
\label{TRevEquiv}
\end{Remark}
Indeed, since $T$ is bijective, we can write $f = Tf'$ and $g = Tg'$. If property (\ref{TDefinition}a) holds for $f$, $g$ and $T^{-1}$, after substituting $Tf'$ and $Tg'$, we come to $Tf' \leq Tg' \Rightarrow f' \leq \tilde{C}g'$. So condition (\ref{TDefinition}a) on $T^{-1}$  is equivalent to condition (\ref{TRevEquiv}a) on $T$. The same applies for (\ref{TRevEquiv}b) and (\ref{TDefinition}b).  \\ \\
Notice that when $\tilde{C}=1$, we have order preserving transform. We would like to show that order preserving transforms are stable, i.e. almost order preserving transforms  are, in some sense, close to the order preserving transforms discussed above. Our main theorems are the following:

\begin{thm}
Let $n \geq 2$. Any $1-1$ and onto transform $T:Cvx_0(\R^n) \rightarrow Cvx_0(\R^n)$ such that both, $T$ and $T^{-1}$ are $\tilde{C}$-almost order preserving, satisfies one of the following conditions: \\
Either
\begin{enumerate}
\item[\ref{MainTheorem}a.]for all $f \in Cvx_0(\R^n), \quad cf\circ B \leq Tf  \leq Cf \circ B$,
\end{enumerate}
or
\begin{enumerate}
\item[\ref{MainTheorem}b.] for all $f \in Cvx_0(\R^n), \quad c(\J f)\circ B \leq  Tf  \leq C(\J f)\circ B$, \\
\end{enumerate}
where $B \in GL(n)$ and $c,C$ are positive constants depending only on $\tilde{C}$.
\label{MainTheorem}
\end{thm}
\begin{Remark}
Actually the proof gives $C \leq \lambda\tilde{C}^7$, but it is entirely possible that the dependence on $\tilde{C}$ is linear.
\end{Remark}
 The "dual" of the above statement follows:
\begin{thm}
Let $n\geq 2$. Any bijective transform $T:Cvx_0(\R^n) \rightarrow Cvx_0(\R^n)$ such that both, $T$ and $T^{-1}$ are almost order reversing, satisfies one of the following conditioins: \\
Either
\begin{enumerate}
\item[\ref{MainTheorem_Dual}a.]{For all $f \in Cvx_0(\R^n), \quad c(\mathcal{A}f)\circ B \leq Tf  \leq C(\mathcal{A}f) \circ B$ }
\end{enumerate}
or
\begin{enumerate}
\item[\ref{MainTheorem_Dual}b.]{For all $f \in Cvx_0(\R^n), \quad c(\mathcal{L}f)\circ B \leq Tf  \leq C(\mathcal{L}f) \circ B $,}
\end{enumerate}
\label{MainTheorem_Dual}
where $B \in GL(n)$ and $c,C$ are positive constants as above.
\end{thm}
In the case of general positive convex functions ($Cvx_{+}(\R^n)$), we have a similar theorem:
\begin{thm}
Let $n \geq 2$. Any bijective transform $T:Cvx_{+}(\R^n) \rightarrow Cvx_{+}(\R^n)$ such that both, $T$ and $T^{-1}$ are almost order preserving, must be close to the identity transform:
\[
cf(Bx + b_0) \leq (Tf)(x)  \leq Cf(Bx + b_0) \\
\]
where $B \in GL(n), b_0 \in \R^n$ and $c,C$ are positive constants.
\label{GeneralCvxStability}
\end{thm}
Notice that there there is no dual statement for the class of general convex non-negative functions, since there exist no order reversing transformations on this class, as was noted by Artstein-Avidan and Milman in \cite{AM2010}.

\section{Preliminaries and Notations}
Let us state that throughout the article, all the constants $c, C, c', C'$ etc, mostly depend on $\tilde{C}$ which appears in the definition of order almost preserving transforms. The dependence is some power of $\tilde{C}$ which can be seen during the proofs. These constants are not universal and might have a dif{}ferent meaning in dif{}ferent context. \\ \\
We will use the notation of {\em convex indicator} functions, $1_K^{\infty}$ where $K$ is some convex domain. As our discussion is limited to convex functions, we define it in the following way:
\[
1_K^{\infty}(x)= \begin{cases} 0, & x  \in K \\ +\infty, & x  \not \in K \end{cases}
\]
 Likewise, we will use {\em modified Delta} functions denoted by $D_\theta + c$, which equals $c$ when $x=\theta$ and $+\infty$ otherwise.
\\ \\
Next we state a known stability result by Hyers and Ulam \cite{Hyers, Ulam}, which we will use in some of the proofs:
\begin{thm}
Let $E_1$ be a normed vector space, $E_2$ a Banach space and suppose that the mapping $f:E_1\rightarrow E_2$ satisfies the inequality
\[
||f(x+y)-f(x)-f(y)|| \leq \epsilon
\]
for all $x, y \in E_1$, where $\epsilon > 0$ is a constant. Then the limit
\[
g(x) = \lim_{n\rightarrow\infty}2^{-n}f(2^nx)
\]
exists for each $x \in E_1$, and $g$ is the unique additive mapping satisfying
\[
||f(x)-g(x)|| \leq \epsilon
\]
for all $x \in E_1$. If $f$ is continuous at a single point of $E_1$, then $g$ is continuous everywhere.
\label{Hyers_Ulam}
\end{thm}
\begin{Lemma}
Assume we have function $f:\R^+ \rightarrow \R^+$, which satisfies the following condition of $C-$monotonicity:
\begin{equation}
x \leq y \text{ implies } f(x) \leq Cf(y) \label{Cmonotonicity}
\end{equation}
for a constant $C > 1$ independent of $x$ and $y$. Then there exist a monotonic function $g(x)$ such that $C^{-1}g(x) \leq f(x) \leq g(x)$.
\label{almost_monotonic_lemma}
\end{Lemma}
\begin{proof} Define $g(x)$ to be the infimum over all monotone functions which are greater or equal $f(x)$:
\[
g(x)=\sup_{0 \leq y \leq x}f(y).
\]
Obviously $g(x) \geq f(x)$ and $g(x)$ is monotone. For any $x_0$, we know that if $y \leq x_0$ then $f(y) \leq Cf(x_0)$. Therefore this is true after applying $\sup$, which brings us to $g(x_0) \leq Cf(x_0)$ as desired.
\end{proof}
\begin{Lemma}
Assume we have a function $f:\R^n \times \R^+ \rightarrow \R^+$ which satisfies the following inequalities for all $(x,a)$ and $(y,b)$:
\begin{equation}
\frac{1}{C}(\lambda f(x,a) + (1-\lambda)f(y,b)) \leq f(\lambda (x,a) + (1-\lambda)(y,b)) \leq C(\lambda f(x,a) + (1-\lambda)f(y,b)),
\end{equation}
and $f(x,0) = 0$, for all $x \in \R^n$. Then there exists a constant $C'$ such that
\[
\frac{1}{C'} a \leq f(x,a) \leq C' a
\]
\label{Triangle_Height_Function_Lemma}
\end{Lemma}
\begin{proof} First we check the case where $n=0$. Substitute $y=0$ and use the fact that $f(0)=0$ to conclude:
\begin{equation}
c\lambda f(a) \leq f(\lambda a) \leq C\lambda f(a).
\label{temp3}
\end{equation}
This is true for every $0 \leq\lambda < 1$ and $a \in \R$, so choose $a=1$ to get almost-linearity of $f$:
\begin{equation}
c\lambda \leq f(\lambda) \leq C\lambda.
\end{equation}
Note that this is true for $\lambda \leq 1$, but we can easily conclude it for all $\lambda$ by taking $a'=\lambda a$ and applying (\ref{temp3}) again. First, rewrite (\ref{temp3}) in the form
\[
\frac{1}{C\lambda} \leq \frac{f(a)}{f(\lambda a)} \leq \frac{1}{c \lambda}.
\]
After substituting $a'=\lambda a$, it becomes:
\[
\frac{1}{C\lambda} \leq \frac{f(a'/\lambda)}{f(a')} \leq \frac{1}{c \lambda}.
\]
Now, if $a'=1$, we get what we required:
\[
\frac{c}{\lambda} \leq f(1/\lambda) \leq \frac{C}{\lambda}.
\]
Replace $1/\lambda$ with $t>1$ and conclude the proof.
To prove the general case, notice that if $x=(x_1, x_2, \ldots x_n)$, then using the previous case
\begin{eqnarray*}
f(x,a) &=& f(\frac{1}{2}(2x_1, 2x_2, \ldots 2x_n-1, 0) + \frac{1}{2}(0, 0, \ldots, 0, 1, 2a))\\
 &\leq& \frac{1}{2}Cf(0,0, \ldots, 1, 2a) \leq C'a
\end{eqnarray*}
In the same way, we see that $f(x,a) \geq \frac{1}{C'}a$. This completes the proof.
\end{proof}

\section{Stability On the Class of Geometric Convex Functions}

\subsection{Preservation of $\sup$ and $\hat{\inf}$}
Since we work with convex functions, taking supremum results in a convex function in our class. However, infimum of convex functions is not necessarily convex, thus we use a modified infimum denoted by $\hat{\inf}$, defined as follows:
\[
\hat{\inf_\alpha}(f_\alpha) = \sup_g (g \in Cvx_0(\R^n) : g \leq f_\alpha \text { for each } \alpha).
\]
Now we can see how almost order preserving transforms act on $\sup$ and $\hat{\inf}$.
\begin{Lemma}If $T$ is almost order preserving transformation, then:
\begin{equation}
\tilde{c}^2T(\max f_\alpha) \leq  \max Tf_\alpha \leq \tilde{C}T(\max f_\alpha) \label{sup_stab}
\end{equation}
\begin{equation}
\tilde{c}T(\hat{\inf}{f_\alpha}) \leq \hat{\inf{Tf_\alpha}} \leq \tilde{C}^2T(\hat{\inf}{f_\alpha}) \label{inf_stab}
\end{equation}
\label{InfSupLemma}
\end{Lemma}
\begin{proof} Since $T$ is bijective, we may assume that there exists such a function $h$ such that: $\max{Tf_\alpha} = \tilde{c}Th.$
Hence, for each $\alpha$, $Tf_\alpha \leq \tilde{c}Th$. Condition (\ref{TRevEquiv}b) implies that $f_\alpha \leq h$ for all $\alpha$. If we define
$h' = \max{f_\alpha}$, we may write $h' \leq h$. Applying condition (\ref{TDefinition}a) we conclude that $Th' \leq \tilde{C}Th$, which means (by definition of $Th$) that $\tilde{c}^2T(\max{f_\alpha}) \leq \max{Tf_\alpha}$. To get the right hand side, we write that $f_\alpha \leq h'$, so by condition (\ref{TDefinition}a) we get that $Tf_\alpha \leq \tilde{C}Th'$. This is true for all $\alpha$, so $\max{Tf_\alpha} \leq \tilde{C}Th'$. The proof of inequality (\ref{inf_stab}) is similar.
\end{proof}

\subsection{Preservation of Zero and Infinity}
\begin{Lemma} \label{lem-ZeroP}
If $T$ is almost order preserving transformation, then $T1_{\{0\}}^{\infty} = 1_{\{0\}}^{\infty}$ and $T0 = 0$.
\end{Lemma}
\begin{proof} Since $1_{\{0\}}^{\infty}$ is the maximal function on the set $Cvx_0(\R^n)$, we may write: $f \leq \tilde{c}1_{\{0\}}^{\infty}$. Using condition (\ref{TDefinition}b), we get $Tf \leq T1_{\{0\}}^{\infty}$, for every $f$. Since $T$ is bijective, $T1_{\{0\}}^{\infty}$ must be the maximal function. In the same way $T0 = 0$.
\end{proof}
\subsection{Ray-wise-ness}
\begin{Lemma}
If $T:Cvx_0(\R^n) \rightarrow Cvx_0(\R^n)$ is an almost order preserving transformation then there exists some bijection $\Phi:S^{n-1} \rightarrow S^{n-1}$ such that a function supported on the ray $\R^+y$ is mapped to a function supported on the ray $\R^+\Phi(y)$.
\label{Raywiseness}
\end{Lemma}
\begin{proof} Let us check that if $\max(f,h) = 1_{\{0\}}^{\infty}$, then $\max(Tf, Tg) = 1_{\{0\}}^{\infty}$. This follows immediately from the fact that $\max(Tf, Tg) \geq \tilde{c}^2T(\max(f,g)) = 1_{\{0\}}^{\infty}$ according to the previous lemma.
The proof of the lemma follows exactly in the same way as in \cite{AM2010}.
\end{proof}

\begin{Lemma} \label{lem-SupportOnRay}
Let $f \in Cvx_0(\R^n)$ and $y \in S^{n-1}$. Then, $(supp Tf)\cap \R^+\Phi(y) = supp T(\max(f, 1_{\R^+y}^\infty))$.
\end{Lemma}
\begin{proof}
Notice that $T(1_{\R^+y}^\infty) = 1_{\R^+\Phi(y)}$. Indeed, the function $1_{\R^+y}^\infty$ is supported on a ray, and thus must be mapped to a function supported on a ray. In addition, it is the smallest function on $\R^+y$ which, by the reasoning of lemma \ref{lem-ZeroP} must be mapped to the smallest function on $\R^+\Phi(y)$. \\
By lemma \ref{InfSupLemma} we have
\[
\tilde{c}^2\max(Tf, 1_{\R^+\Phi(y)}^\infty) \leq T(\max(f, 1_{\R^+y}^\infty)) \leq \tilde{C}\max(Tf, 1_{\R^+\Phi(y)}^\infty).
\]
Hence, we see that on the ray $R^+\Phi(y)$, the function $Tf$ is finite if and only if the function $T(\max(f, 1_{\R^+y}^\infty))$. This completes the proof.
\end{proof}

\subsection{Convex Functions on $\R^+$}
We have seen that due to the ray-wise-ness, the case of $\R^+$ will give us an idea about the general case. We  will state and proof this special case, which is actually not required for the general one, but is of independent interest.
\begin{thm}
if $T$ and $T^{-1}$ are both $\tilde{C}$-almost order preserving transforms on the class of convex
geometric functions $Cvx_0(\R^+)$ and $T$ is bijective, then there exist positive
constants $\alpha_1$, $\alpha_2$, $\beta_1$, $\beta_2$ (dependent of
$\tilde{C}$) , such that either
\begin{enumerate}
\item[\ref{Stability_Thm}a.] {for all  $f \in Cvx_0(\R^+), \quad \beta_1f(x/\alpha_2) \leq Tf(x) \leq \beta_2f(x/\alpha_1)$, }
\end{enumerate}
or
\begin{enumerate}
\item[\ref{Stability_Thm}b.] {for all  $f \in Cvx_0(\R^+), \quad \alpha_1\J f(x/\beta_2) \leq Tf(x) \leq \alpha_2 \J f(x/\beta_1).$}
\end{enumerate}
\label{Stability_Thm}
\end{thm}
The proof of this theorem uses a few preliminary constructions and facts which we introduce and study in sections (\ref{PTildeSection}-\ref{TrianglesSection}).
\subsection{Property $\tilde{P}$} \label{PTildeSection}
To study general functions, we need a family of extremal functions which are easy to deal with and can be used to describe the general case. In the case of geometric convex functions the family of indicators and linear functions is most convenient. A property which uniquely defines such a family was introduced in \cite{AM2010} (property $P$). Due to the modified nature of our problem, we introduce a slightly modified property:
\begin{Definition} \label{def-PropertyPTilde} We say that a function $f$ satisfies property $\tilde{P}$, if there exist no two functions $g,h \in\ Cvx_0(\R^n)$, such that $g \ngeq \tilde{c}^3f$, $h \ngeq \tilde{c}^3f$ but $\max(g, h) \geq f$.
\end{Definition} Note that Definition \ref{def-PropertyPTilde} depends on the constant $\tilde{C}$. \\
Obviously, if a function satisfies property $P$, then it satisfies property $\tilde{P}$. This means that the family of functions which satisfy $\tilde{P}$ contains all the indicator functions and linear functions through $0$.
We will now show that the non-linear functions that have property $\tilde{P}$ cannot dif{}fer greatly from the linear ones, unless they are indicators.

\begin{Lemma} \label{lem-PTildeZeroDerivative}
If $f$ has property $\tilde{P}$ and $f$ is not an indicator, then $f'(0) > 0$.
\end{Lemma}
\begin{proof}
Assume otherwise, that is, $f'(0) = 0$. Since $f$ is not an indicator, there exists a $x_0 > 0$ such that $f(x_0) > 0$. Define
\[
L(x) = \frac{f(x_0)}{x_0} x,
\]
and $L_2(x):=\tilde{c^3}L_1(x)$.
Since $f'(0) = 0$, there exists a point $x_1$ such that $L_2(x_1) = f(x_1)$. Hence, it holds that $\max(L_1, 1_{[0,x_1]}^\infty) \geq f$. Since $f$ has property $\tilde{P}$ we have that either $1_{[0,x_1]}^\infty \geq \tilde{c}^3f$ or $L_1 \geq \tilde{c}^3f$. Clearly, neither of the inequalities holds and this is a contradiction to the fact that $f$ has property $\tilde{P}$.
\end{proof}

\begin{Lemma}
Every function with property $\tilde{P}$ that is not an indicator can be bounded by
\[
f'(0)z \leq f(z) \leq \tilde{C}^3 f'(0)z.
\]
\label{TildeP_Bounds}
\end{Lemma}
\begin{proof} First note that by Lemma \ref{lem-PTildeZeroDerivative} $f'(0) > 0$. It is clear that $f(z)$ is bounded by $L_1(z) := f'(0)z$ from below. Assume that $L_2(z) := \tilde{C}^3L_1(z)$ intersects $f(z)$ at some point $x_0$.  This means that $L_1(z)$ intersects $\tilde{c}^3f(z)$ at $x_0$. Hence, the derivative of $\tilde{c}^3f(x)$ at $x_0$ is bigger than $f'(0)$ (otherwise, they wouldn't intersect). So there exists a constant $a > f'(0)$ such that $L_a(z) = az$ intersects both, $f(z)$ and $\tilde{c}^3f(z)$. This is a contradiction to property $\tilde{P}$ (take $L_a$ and an indicator function to see this).
\end{proof}
\begin{Lemma}
If a function $f$ with property $\tilde{P}$, equals $\infty$ for $x \geq x_0$ for some $x_0$,
it must be an indicator.
\end{Lemma}
\begin{proof} Assume there exists $x_1 < x_0$ such that $f(x_1) =
c > 0$. Then, define two functions: $g(x) = \frac{c}{x_1}x$ and
$h(x) = 1_{[0, x_1]}^{\infty}$. Obviously, $g(x) \ngeq \tilde{c}^3f(x)$ and
$h(x) \ngeq \tilde{c}^3f(x)$, but $\max(g,h) \geq f$, which
contradicts $f$ having property $\tilde{P}$.
\end{proof}
From now on, a function $f$ with property $\tilde{P}$ that is not an
indicator, will be called \textit{almost linear} function.

\subsection{Properties of $\tilde{P}$}
\begin{Lemma}If $T$ is almost order preserving, and $f$ has property $P$, then $Tf$ has property $\tilde{P}$.
\label{PToTildeP}
\end{Lemma}
\begin{proof} Assume that $f$ has property $P$, but $Tf$ does not have $\tilde{P}$. So there exist $g, h$ such that $g \ngeq \tilde{c}^3 Tf$, $h \ngeq \tilde{c}^3 Tf$, but $\max(g,h) \ge Tf$. Since $T$ is bijective, there exist $\phi, \psi \in Cvx_0(\R^n)$ such that
\[
g = \tilde{c}^2 T\phi, \qquad h = \tilde{c}^2T\psi.
\]
Using property (\ref{sup_stab}) we write:
\[
\tilde{c}T(\max(\phi, \psi)) \geq \max(g, h) \geq Tf.
\]
Now, applying condition (\ref{TRevEquiv}b), we conclude that $f \leq
\max(\phi, \psi)$. Since $f$ has property $P$, then either $f
\leq \phi$ or $f \leq \psi$. After applying condition
(\ref{TDefinition}a), we get that either $Tf \leq \tilde{C}T\phi =
\tilde{C}^3 g$, or $Tf \leq \tilde{C^3} h$, which is a
contradiction. The same proof applies for $T^{-1}$, so $T^{-1}$ also maps functions with property $P$, to functions with property $\tilde{P}$.
\end{proof}

\begin{Lemma}
$T$ either maps all indicators to indicators or it maps all
indicators to almost linear functions.
\end{Lemma}
\begin{proof}
Assume that the claim is false and there exist two indicators $1_{[0,x]}^\infty, 1_{[0, y]}^\infty$ which are mapped to an indicator $1_{[0,x']}^\infty$ and an almost linear function $f$, respectively. Assume, without loss of generality, that $x < y$. Then we know that $1_{[0,y]}^\infty \leq \tilde{c}1_{[0, x]}^\infty$, which by properties of $T$ implies that $f$ and $1_{[0,x']}^\infty$ are comparable. But we know that an almost linear function cannot be comparable to an indicator, and the proof is complete.
\end{proof}

\begin{Lemma} $T$ either maps all linear functions to
almost linear functions, or it maps them to the indicators.
\end{Lemma}
\begin{proof}
Assume that there exist linear functions $l_a, l_b$ such that $T(l_a) = 1_{[0,x]}^\infty$ and $T(l_b)=f$ where $f$ is almost linear. In addition, assume, without loss of generality, that $l_a < l_b$. Then, by properties of $T$ we have $T(l_a) < \tilde{C}T(l_b)$, which is equivalent to $1_{[0,x]}^\infty \leq \tilde{C}f$. But this is a contradiction since indicators and almost linear functions are not comparable.
\end{proof}


\begin{Lemma}
$T$ cannot map all functions with property $P$ to indicators, and it
cannot map all functions with property $P$ to almost linear
functions.
\end{Lemma}
\begin{proof} Assume, all functions are mapped to indicators.
Then we may write that $T(1_{[0, x]}^{\infty}) = 1_{[0, y]}^{\infty}$ and $T(l_a) =
1_{[0, z]}^{\infty}$. If, without loss of generality, $y < z$, then $T(l_a) \leq
\tilde{c}T(1_{[0, x]}^{\infty})$. This implies that $l_a$ and $1_{[0, z]}^{\infty}$
are comparable, which cannot be. \\
The other option, is that all functions with property $P$, are
mapped to almost linear functions: $T1_{[0, x]}^{\infty} = f$,
and $T(l_a) = g$. Since both $f$ and $g$ are almost linear, there exists a linear function
$l_b$ such that $f \leq \tilde{c}l_b$ and $g \leq \tilde{c}l_b$. By lemma \ref{PToTildeP}, we know that the function $T^{-1}(l_b)$ is either almost linear or an indicator. If it is an indicator then the inequality $g \leq \tilde{c}l_b$ implies that $l_a \leq T^{-1}(l_b)$ (property \ref{TRevEquiv}b). This is a contradiction as an indicator cannot be comparable to linear function. If it is almost linear, then the inequality $f \leq \tilde{c}l_b$ implies that $1_{[0,x]}^{\infty} \leq T^{-1}(l_b)$. This is again a contradiction since almost linear functions cannot be comparable to indicators. This completes the proof.
\end{proof}
\begin{Lemma}If $T$ maps linear functions to indicators, $T$ preserves property $P$.
\end{Lemma}
\begin{proof}
Assume that we are in the case where indicators are mapped to almost-linear functions, and linear functions are mapped to indicators. First we will show that $T$ maps the linear functions onto the indicators. If we have a $g$ which is not linear, but $Tg = 1_{[0,z]}^{\infty}$, then we can find a linear function $ax$ that intersects $g$. $l_a$ is mapped to some indicator $1_{[0,y]}$. If $y > z$ then $1_{[0,y]}^{\infty} < \tilde{c}1_{[0,z]}^{\infty}$, hence $g$ and $l_a$ are comparable. This is a contradiction, so $g$ must be linear. \\
Now we know that the indicators are mapped to almost-linear functions under $T$. But $T^{-1}$ maps property $P$ to $\tilde{P}$, so if we restrict it to linear functions, we get only indicators. So the preimage of every linear function is an indicator. But this also means that there are no non-linear functions with property $\tilde{P}$ other than indicators. The reason is similar: If we have a non-linear function $g=T1_{[0,y]}^{\infty}$, then intersect it with some linear function $l_b$. They are not comparable, but the sources are, which is a contradiction. This ends the proof. \\ \\
\end{proof}
\begin{Lemma}If $T$ maps linear functions to indicators, then it is order preserving on functions that have property $P$. \label{Order_Preserving_On_P}
\end{Lemma}
\begin{proof} Assume that $T(1_{[0,z]}^{\infty})= l_{\phi(z)}$ and $T(l_a) = 1_{[0,c(a)]}^{\infty}$.
If $1_{[0,z]}^{\infty} < 1_{[0,z']}^{\infty}$ ($z' < z$), then $1_{[0,z]}^{\infty} < \tilde{c}1_{[0,z']}^{\infty}$ and using property (\ref{TDefinition}a) we write $\phi(z) < \phi(z')$. So we see that $\phi(z)$ is monotone decreasing and continuous. The same applies for $c(a)$.
\end{proof}
\subsection{Triangles} \label{TrianglesSection}
Define triangle:
$\lhd_{z,c } = \max(l_c, 1_{[0,z]}^{\infty}).$ Using facts shown above, we conclude that:
\[
\tilde{c} \lhd_{z', c'} \leq  T(\lhd_{z,c}) \leq
\tilde{C}^2\lhd_{z', c'}
\]
Where $z' = z'(z,c)$ and $c' = c'(z,c)$.
Next, we show that triangles determine everything, in the general setting of $Cvx_0(\R^n)$. To this end, we will need a definition of triangle in $\R^n$. Given a vector $z$ and a gradient $c$, define:
\[
\lhd_{z, c}(x) = \max \{ c\frac{x}{|z|}, 1_{[0,z]}^\infty \}.
\]
\begin{Lemma}
Assume that $T$ is an almost order preserving map defined on $Cvx_0(\R^n)$, and that there exists $B \in GL(n)$ and a constant $\beta$ such that
\begin{equation}
\tilde{c}\beta \lhd_{B^{-1}z, d} \leq T(\lhd_{z,d}) \leq \tilde{C}\beta \lhd_{B^{-1}z, d}.
\label{I_Triangle_Inequality}
\end{equation}
Then, $T$ is almost a variation of identity, in the following sense:
\[
\frac{1}{C}\varphi(Bx) \leq (T\varphi)(x) \leq C\varphi(Bx)
\]
\label{I_Triangle_Lemma}
\end{Lemma}
\begin{proof} Any $\varphi(x) \in Cvx_0(\R^n)$ can be written as $\varphi(x) = (\hat{\inf_y}\lhd_{y,\varphi(y)})(x)$. Hence, using (\ref{InfSupLemma}) we have $(T\varphi)(x) \leq \tilde{C}^2(\hat{\inf_y}T(\lhd_{y, \varphi(y)})(x)$. Using assumption (\ref{I_Triangle_Inequality}) we conclude that
\[
(T\varphi)(x) \leq \tilde{C}^3\beta (\hat{\inf_y}(\lhd_{B^{-1}y, \varphi(y)}))(x) = \tilde{C}^3\beta (\hat{\inf_z}\lhd_{z, \varphi(Bz)})(x) = \tilde{C}^3\beta\varphi(Bx).
\]
The other part is concluded in a similar way.\\
\end{proof}
\begin{Lemma}
Assume that $T$ is an almost order preserving map defined on $Cvx_0(\R^n)$, and that there exist $B \in GL(n)$ and a constant $\beta$ such that
\begin{equation}
\tilde{c}\beta \lhd_{\frac{B^{-1}z}{d|z|}, \frac{1}{|z|}} \leq T(\lhd_{z,d}) \leq \tilde{C}\beta \lhd_{\frac{B^{-1}z}{d|z|}, \frac{1}{|z|}}.
\label{J_Triangle_Inequality}
\end{equation}
Then, $T$ is almost a variation of $\J$, in the following sense:
\[
\frac{1}{C}(\J\varphi)(Bx) \leq (T\varphi)(x) \leq C(\J\varphi)(Bx)
\]
\label{J_Triangle_Lemma}
\end{Lemma}
The proof is similar to lemma \ref{I_Triangle_Lemma}.

\subsection{Proof of Theorem (\ref{Stability_Thm})}
\subsubsection{The Case of "J"}
We stay with the notations of lemma \ref{Order_Preserving_On_P}. First we deal with the case we identify with $\J$, i.e linear functions are mapped to indicators and vice versa. Define $g = \hat{\inf}\{1_{[0,tz]}^{\infty}, l_{a/(1-t)}\}$, for some $0 < t < 1$. Now, $g \leq \lhd_{z,a}$, but it does not hold for any $a' < a$ or $z' > z$. Applying $T$ to the last inequality we get that $T(g) \leq \tilde{C}T(\lhd_{z,a})$. Using property (\ref{inf_stab}):
\begin{equation}
\tilde{c}^2\hat{\inf}(1_{[0,c(a/(1-t)]}^{\infty}, l_{\phi(tz)}) \leq T(g) \leq \tilde{C}\hat{\inf}(1_{[0,c(a/(1-t)]}^{\infty}, l_{\phi(tz)}),
 \label{inf_triangle_bounds}
\end{equation}
and using (\ref{sup_stab}) for the triangle:
\begin{equation}
\tilde{c}\lhd_{c(a), \phi(z)} \leq T(\lhd_{z,a}) \leq \tilde{C}^2\lhd_{c(a), \phi(z)}.
\label{sup_triangle_bounds}
\end{equation}
Plugging (\ref{inf_triangle_bounds}) and (\ref{sup_triangle_bounds}) into our inequality, we conclude that
\begin{equation}
\tilde{c}^2\hat{\inf}(1_{[0,c(a/(1-t)]}^{\infty}, l_{\phi(tz)}) \leq \tilde{C}^2\lhd_{c(a), \phi(z)},
\end{equation}
which means that
\begin{equation}
(c(a) - c(a/(1-t)))\phi(tz) \leq \tilde{C}^4c(a)\phi(z). \label{HS_Left}
\end{equation}
We know that $g \nleq \lhd_{z,a'}$ for $a' < a$. This means that $T(g) \nleq \tilde{c}T(\lhd_{z,a'})$. Using the right-hand side of (\ref{inf_triangle_bounds}) and the left hand side of (\ref{sup_triangle_bounds}) we conclude
\begin{equation}
\tilde{c}^3c(a')\phi(z) \leq (c(a) - c(a/(1-t)))\phi(tz) \label{HS_Right}.
\end{equation} This is true for every $a' < a$, so using continuity, we may right the above with $a$ instead.
Inequalities (\ref{HS_Left}) and (\ref{HS_Right}) can be rewritten together:
\begin{equation}
\tilde{c}^3\frac{c(a)}{c(a)-c(a/(1-t))} \leq \frac{\phi(tz)}{\phi(z)} \leq \tilde{C}^4\frac{c(a)}{c(a)-c(a/(1-t))}. \label{temp1}
\end{equation}
Choose $z=1$ to get:
\begin{equation}
\tilde{c}^3\frac{c(a)}{c(a)-c(a/(1-t))} \leq \frac{\phi(t)}{\phi(1)} \leq \tilde{C}^4\frac{c(a)}{c(a)-c(a/(1-t))}. \label{temp2}
\end{equation}
Combining (\ref{temp1}) and (\ref{temp2}), we come to the inequality
\begin{equation}
c\phi(t)\phi(z) \leq \phi(tz) \leq C\phi(t)\phi(z), \label{PHI_SOLVE}
\end{equation}
for some constant $C$ and $c = C^{-1}$. To solve this, substitue $t = \exp{\alpha_1}$ and $z=\exp{\alpha_2}$, and define $h(s) = \log (\phi(e^s))$. Then after applying $\log$, (\ref{PHI_SOLVE}) becomes:
\begin{equation}
-C' + h(\alpha_1) + h(\alpha_2) \leq h(\alpha_1 + \alpha_2) \leq C' + h(\alpha_1) + h(\alpha_2),
\end{equation} or equivalently:
\begin{equation}
|h(\alpha_1 + \alpha_2) - h(\alpha_1) - h(\alpha_2)| \leq C'.
\end{equation}
The Hyers-Ulam theorem (\ref{Hyers_Ulam}), implies that there exists a linear $g(\alpha) = \gamma\alpha$, such that $|h(\alpha) - g(\alpha)| < C'$. This means that $|\log(\phi(e^s)) - \log(e^{\gamma s})|$ and it is easy to check that this implies the following on $\phi$:
\begin{equation}
c''z^{\gamma} \leq \phi(z) \leq C''z^{\gamma}, \qquad c'' = \frac{1}{C''}
\end{equation}
Notice that $\gamma < 0$, since we know that $\phi$ is decreasing.
Let use now proceed to estimate $c(a)$.  Notice that all the arguments applied so far can be reused for $T^{-1}$, hence there exists $\gamma' < 0$ and $c'z^{\gamma'} \leq \psi(z) \leq C'z^{\gamma'}$ such that $T^{-1}1_{[0,z]}^{\infty}  = l_{\psi(z)}$. We know that $Tl_a = 1_{[0, c(a)]}^{\infty}$, or $T^{-1}1_{[0, c(a)]}^{\infty} = l_a$, which is equivalent. But we have just shown that $T^{-1}1_{[0, c(a)]}^{\infty} = l_{\psi(c(a))}$. So,
\begin{equation}
c'(c(a))^{\gamma'} \leq a = \psi(c(a)) \leq C'(c(a))^{\gamma'}.
\label{J_ca_inequality}
\end{equation}
Rewriting (\ref{J_ca_inequality}) gives
\begin{equation}
c'a^{1/\gamma'} \leq c(a) \leq C'a^{1/\gamma'}.
\label{J_ca_bounds}
\end{equation}
Using (\ref{HS_Left}), (\ref{HS_Right}) and the estimate we have for $\phi(z)$, we write:
\begin{equation}
c''^2\tilde{c}^3 t^{-\gamma} \leq  \tilde{c}^3\frac{\phi(z)}{\phi(tz)} \leq  \frac{c(a)-c(a/(1-t))}{c(a)} \leq \tilde{C}^4\frac{\phi(z)}{\phi(tz)} \leq C''^2\tilde{C}^4t^{-\gamma},
\end{equation}
which is equivalent to
\begin{equation}
1 - C''' t^{-\gamma} \leq \frac{c(a/(1-t))}{c(a)} \leq 1 - c''' t^{-\gamma}. \label{C_SOLVE}
\end{equation}
Choose $a=1$ to get the following:
\begin{equation}
1 - C''' t^{-\gamma} \leq \frac{c(1/(1-t))}{c(1)} \leq 1- c''' t^{-\gamma}. \label{C_SOLVE}
\end{equation}
Using the bounds we got in (\ref{J_ca_bounds}), we see that $\gamma = \gamma' = -1$.
To summarize, there exist positive constants $\alpha_1$, $\alpha_2$ and $\beta_1$, $\beta_2$, such that
\begin{equation}
\frac{\alpha_1}{z} \leq \phi(z) \leq \frac{\alpha_2}{z}, \qquad \frac{\beta_1}{a} \leq c(a) \leq \frac{\beta_2}{a}
\end{equation}
To conclude the proof, notice that we have:
\begin{equation}
T(\lhd_{z,a}) \leq \tilde{C}^2\lhd_{\beta_1/a, \alpha_2/z} =
\tilde{C}^2\J (\lhd_{z/\alpha_2, a/\beta_1}).
\end{equation}
Also, notice that $f(x) = (\hat{\inf}_y \lhd_{y,f(y)})(x)$, so
\begin{eqnarray*}
(Tf)(x) \leq \tilde{C}\hat{\inf}T(\lhd_{y,f(y)}) &\leq& \tilde{C}^3
\hat{\inf}(\lhd_{\beta_1/f(y), \alpha_2/y})  \\
=\alpha_2\tilde{C}^3\hat{\inf}\J (\lhd_{y, f(y)/\beta_1}) &=&
\alpha_2\tilde{C}^3\J (\hat{\inf}\lhd_{y, f(y)/\beta_1}) =
\tilde{C}^3\alpha_2\J (\frac{1}{\beta_1}f(x)) = C'\J(f(x/\beta_1).
\end{eqnarray*}
The same applies for the lower bound.

\subsubsection{The Case of "I"}
In the case $T$ maps indicators to themselves, we do not know that it preserves property $P$. Assume now that $T1_{[0,z]}^{\infty} = 1_{[0,\phi(z)]}^{\infty}$. We also know, due to lemma \ref{TildeP_Bounds}, that if $T(l_a) = f$, then $f'(0)x \leq f(x) \leq \tilde{C}^3f'(0)x$. If we define $c(a) = f'(0)$, then, $l_{c(a)} \leq T(l_a) \leq \tilde{C^3}l_{c(a)}$. Now, we can estimate how triangles are mapped:
\begin{equation}
T(\lhd_{z,a}) \leq \tilde{C}^2\max(1_{[0, \phi(z)]}^{\infty}, T(l_a)) \leq \tilde{C}^5 \lhd_{\phi(z), c(a)},
\end{equation} and
\begin{equation}
T(\lhd_{z,a}) \geq \tilde{c}\max(1_{[0, \phi(z)]}^{\infty}, T(l_a)) \geq \tilde{c} \lhd_{\phi(z), c(a)}.
\end{equation}
Now, let us rewrite the bounds for $g$, from the previous case:
\begin{equation}
T(g) \leq \tilde{C}\hat{\inf}(1_{[0, \phi(z)]}^{\infty}, T(l_{a/(1-t)})) \leq \tilde{C}^4 \hat{\inf}(1_{[0, \phi(z)]}^{\infty}, l_{c(a/(1-t)}),
\end{equation}
and
\begin{equation}
T(g) \geq \tilde{c}^2 \hat{\inf}(1_{[0, \phi(z)]}^{\infty}, l_{c(a/(1-t)}).
\end{equation}
Using the fact that $T(g) \leq \tilde{C}T(\lhd_{z,a})$, we get a similar inequality:
\begin{equation}
(\phi(z) - \phi(tz))c(a/(1-t)) \leq \tilde{C}^7\phi(z)c(a),
\end{equation}
and using the same methods as before (this time taking $z' > z$, since we only know that $\phi$ is continuous), we get the lower bound:
\begin{equation}
(\phi(z) - \phi(tz))c(a/(1-t)) \geq \tilde{c}^6\phi(z)c(a).
\end{equation}
We can see that we have the same inequalities we had in the previous case, but with $\phi$ and $c$ interchanged, and we come to the inequality
\begin{equation}
\frac{1}{C_1}c(1/(1-t)) \leq \frac{c(a/(1-t))}{c(a)} \leq C_1c(1/(1-t)).
\end{equation}
After substituting $s = 1/(1-t)$, we come to an inequality we already know how to solve:
\begin{equation}
c_1c(s)c(a) \leq c(as) \leq C_1c(s)c(a).
\end{equation}
Using Hyers-Ulam thorem again (\ref{Hyers_Ulam}), we conclude again that
$ct^{\gamma} \leq c(t) \leq Ct^{\gamma}$, for some $\gamma$. To find
bounds for $\phi$, substitute this in the original inequality, and
conclude that $\alpha_1z \leq \phi(z) \leq \alpha_2z$. After we have
this estimate it is easy to conclude that $\gamma = 1$. \\To
conclude the proof, notice that we have:
\begin{equation}
T(\lhd_{z,a}) \leq \tilde{C}^5\lhd_{\alpha_1z, \beta_2a}
\end{equation}
Also, notice that $f(x) = (\hat{\inf}_y \lhd_{y,f(y)})$(x), so
\begin{eqnarray*}
(Tf)(x) \leq \tilde{C}\hat{\inf}T(\lhd_{y,f(y)}) &\leq& \tilde{C}^6
\hat{\inf}\lhd_{\alpha_1y, \beta_2f(y)} =
\tilde{C}^6\beta_2f(x/\alpha_1).
\end{eqnarray*}
The same applies for the lower bound.

\subsection{Completing The Proof in $\R^n$}
Recall, that we know there exists a function $\Phi:S^{n-1}\rightarrow S^{n-1}$ (1-1 and onto), such that any function supported on $\R^+y$ is mapped to a function supported on $\R^+\Phi(y)$.
Let us define for each $y \in S^{n-1}$ a number $j(y)$ that equals $0$ if $T$, restricted to $\R^+y$ behaves like the identity (indicators are mapped to indicators) and $1$ if $T$, restricted to $\R^+y$ behaves like $\J$ (indicators are mapped to linear functions and vice versa).

\begin{Lemma}
Denote $S_0 = \{y \in S^{n-1}:j(y)=0\}$, $S_1 = \{y \in S^{n-1}: j(y)=1\}$.
Then, either $S_0=S^{n-1}$, or $S_1=S^{n-1}$.
\end{Lemma}
\begin{proof}
Let us see that $S_i, \Phi(S_i)$ are convex. To see this, consider the function $f=1_B^\infty$ where $B$ is the $n$-dimensional unit ball. Let $x\in S_1$. By lemma \ref{lem-SupportOnRay} we know that the support of $Tf$ on $\R^+\Phi(x)$ is the same as the support of $T(\max(f, 1_{\R^+x}^\infty))$, and the latter is $\R^+\Phi(x)$. Since $Tf$ is a convex function with a convex support, we get that for every $x,y\in S_1$, the support of $Tf$ must contain every ray $\R^+\Phi(z)$ such that $\Phi(z)$ is contained in $\Phi(x)\vee\Phi(y)$. Hence, $\Phi(S_1)$ is convex. Choosing $g(x) = |x|$, by the same argument, we get that $\Phi(S_0)$ is also convex. Since $T$ and $T^{-1}$ have the same properties we may conclude that $S_i$ are also convex.\\
Notice that $S_0 \cup S_1 = S^{n-1}$, so either one of the sets is empty and we are done, or $S_0$ and $S_1$ are both half-spheres, and likewise $\Phi(S_i)$. In this case let us check how $T$ acts on the function $f$. Denote by $H_1$ the half-space $\bigcup_{y\in\Phi(S_1)}\R^+y$, and by $H_0$ the half-space $\bigcup_{y\in \Phi(S_0)}\R^+y$. We know that for $y \in S_1$, $\R^+\Phi(y)\subset supp(Tf)$. Hence, $supp(Tf)$ contains $H_1$. This means that the support of $Tf|_{H_0}$ cannot be bounded. But then, we could choose a convex, bounded set $K \subset H_0$ (containing zero in the interior of the boundary) and consider the pre-image of $1_K^\infty$. Since $T^{-1}$ preserves order on indicators we know that the support $M$ of $T^{-1}(1_K^\infty)$ will be contained in the unit ball $B$. Consider the function $h=1_{M \vee (-M)}^\infty$. By the preceding argument we know that $supp(Th)$ contains $H_1$ and $supp(Th)\cap H_0$ is bounded, which is a contradiction to the fact that $supp(Th)$ is convex.
\end{proof}

Now we proceed to analyze the behavior of $T$ in each case.\\

The case of $j\equiv 0$. Define $\varphi:\R^n \rightarrow \R^n$ by $T1_{[0,x]}^{\infty} = 1_{[0, \varphi(x)]}^{\infty}$. We know that $T$ preserves $\sup$ and $\hat{\inf}$ on indicators with equality (compare to lemma \ref{InfSupLemma}), thus for any convex body $K$ with $0 \in K$, $T1_K^\infty=1_{\varphi(K)}^\infty$, and $\varphi(K)$ is also convex. The point map $\varphi$ therefore induces an order preserving isomorphism on ${\cal K}_0^n$, the class of convex bodies containing the origin, and by known results (see \cite{AM5}), this implies $\varphi$ is a linear.\\

Take two triangles $\lhd_1, \lhd_2$ with bases $x, x'$ and heights $a,a'$ accordingly. The largest triangle $\lhd_\lambda$ which is smaller than $\hat{\inf}(\lhd_1, \lhd_2)$ with the base $\lambda x + (1-\lambda)x'$ has the height $\lambda a + (1-\lambda)a'$. Denote by $h(x, a)$ the height of the maximal triangle which bounds $T(\lhd_1)$ from below. We have shown before that $\tilde{C}^3h(x,a)$ will be the height of the triangle that bounds $T(\lhd_1)$ from above. Since $T(\lhd_\lambda) \leq \tilde{C}^2\hat{\inf}(T\lhd_1, T\lhd_2)$, we may write (using lemma \ref{InfSupLemma}):
\[
h(\lambda (x,a) + (1-\lambda)(x', a')) \leq \tilde{C}^7(\lambda h(x,a) + (1-\lambda)h(x', a')).
\]
Notice that for a given $x$, $h$ satisfies conditions of lemma \ref{almost_monotonic_lemma}. This is verified by choosing $a'=1$. Applying this lemma we know that for every $x$ there exists a monotone function $\omega_x(a)$ such that $\tilde{c}\omega_x(a) \leq h(x,a) \leq \omega_x(a)$.
We know that if we increase the height of the triangle $\lhd_\lambda$ by some $\epsilon > 0$ (denote this triangle by $\lhd_\epsilon$), then
$T(\lhd_\epsilon) \nleq \tilde{c}\hat{\inf}(T\lhd_1, T\lhd_2)$. Hence, by lemma \ref{InfSupLemma} combined with properties of $T$, we have
\begin{equation}
h(\lambda (x,a) + (1-\lambda)(x', a') + (0, \epsilon)) \geq \tilde{c}^5(\lambda h(x,a) + (1-\lambda) h(x', a')).
\label{Height_function_lower_est}
\end{equation}
We would like now to say that the  inequality holds when $\epsilon \rightarrow 0$, but we don't know that $h$ is continuous. We do know however, that in the worst case, the right hand side of (\ref{Height_function_lower_est}) is multiplied by $\tilde{C}$ after taking the limit (due to existence of $\omega_x$ which is monotone and continuous). Hence
\[
h(\lambda (x,a) + (1-\lambda)(x', a')) \geq \tilde{c}^6(\lambda h(x,a) + (1-\lambda) h(x', a')).
\]
Applying lemma \ref{Triangle_Height_Function_Lemma} on $h(x,a)$, we conclude that there exists a constant $\beta$ such that
\[
\tilde{c}\beta a \leq h(x,a) \leq \tilde{C}\beta a.
\]
To sum it up, we know that for a triangle $f$, $\tilde{c}\beta f \circ A \leq T(f) \leq \tilde{C}\beta f \circ A$. Using lemma \ref{I_Triangle_Lemma} we conclude the same inequality for every $f \in Cvx_0(\R^n)$. \\

The case of $j\equiv 1$. In this case we know that lines are mapped to indicators and vice-versa. Notice, that we cannot compose $T$ with $\J$ and apply the previous case, since $\J \circ T$ would not necessarily satisfy the conditions of almost order preserving transform. However, we do know that in this case $\J \circ T$ is order preserving on the extremal family of indicators and rays. Thus, as explained above (the case of $j\equiv 0$), $(\J\circ T)1_{[0,z]}^{\infty}=1_{[0, Bz]}^{\infty}$ for some $B \in GL(n)$. Composing both sides with $\J$ (recall that $\J$ is an involution), we see that $T$ sends the indicator $1_{[0,z]}^{\infty}$ to a ray in direction $Bz$. Due to the ray-wise-ness of the problem we may conclude that any function supported on a ray in direction $z$ is mapped to a function supported by the ray $\R^+Bz$ ($\Phi(z) = Bz$). \\
Take two indicators $I_1 = 1_{[0,z]}^{\infty}$ and $I_2 = 1_{[0,z']}^{\infty}$ and define the function $g = \hat{\inf}(I_1, I_2)$. The indicator $I_\lambda=1_{[0, \lambda z +  (1-\lambda) z']}^{\infty}$ is bigger then $g$, but for every $\epsilon > 0$ the indicator $I_\epsilon=1_{[0, (1+\epsilon)(\lambda z +  (1-\lambda) z')]}^{\infty}$ is not comparable to $g$. Since $T$ is order preserving on indicators and rays, it preserves the $\hat{\inf}$, so $\hat{\inf}(TI_1, TI_2) = Tg \leq TI_\lambda$, but the same is not true for any $T_I\epsilon$. Define $\psi(z)$ by the way $T$ maps indicators to lines: $T1_{[0,z]}^{\infty} = l_{Bz/|z|, \psi(z)}$. Hence the ray $TI_\lambda$ is comparable to the sector $Tg$ that is spanned by the rays $TI_1, TI_2$. Since the same is not true for $T_\epsilon$, and $\psi$ is monotone in every direction, we conclude that $Tg$ is a linear combination of $TI_1$, $TI_2$. Using this fact we come to the following property of $\psi$:
\begin{equation}
\psi(\lambda z + (1-\lambda)z') = \frac{\lambda |z|}{|\lambda z + (1-\lambda)z'|}\psi(z) + \frac{(1-\lambda)|z'|}{|\lambda z + (1-\lambda)z'|}\psi(z').
\end{equation}
Define the function $h(z):=|z|\psi(z)$. It follows that $h(z)$ satisfies: $h(\lambda z + (1-\lambda)z') = \lambda h(z) + (1-\lambda)h(z')$, from which it follows that $h$ is linear: $h(z) = <u_0, z> + \beta$ for some vector $u_0 \in \R^n$ and a constant $\beta$. Since $\psi(z)$ cannot be zero, $u_0 = 0$, it means that $\psi(z) = \beta/|z|$. \\
Now, define $\theta(z,a)$ by $Tl_{z,a}=1_{[0, Bz\theta(z,a)/|z|]}^{\infty}$. Finding $\theta(z,a)$ can be accomplished directly as with $\psi$, but it is simpler to notice that $T$ and $T^{-1}$ have the same properties. On one hand we know that $T^{-1}1_{[0, Bz/|z|\theta(z,a)]}^{\infty} = l_{z,a}$. On the other hand, applying the same arguments used for $\psi$, we have $T^{-1}1_{[0, Bz/|z|\theta(z,a)]}^{\infty} = l_{z/|z|, \gamma/\theta(z,a)}$, for some $\gamma > 0$. Thus,
\[
a = \frac{\gamma}{\theta(z,a)},
\]
or equivalently, $\theta(z,a) = \gamma/a$.
Using lemma \ref{J_Triangle_Lemma} we conclude the theorem. \\ \\
To show the dual statement (\ref{MainTheorem_Dual}), apply $\mathcal{A}$ to $T$, and use the homogeneity of $\mathcal{A}$ to conclude that $T$ is almost order preserving.  Now we know that $\mathcal{A}T$ is either almost-$\J$ or almost identity. Applying $\mathcal{A}$ again, and using the fact that it is an involution and that $\mathcal{AJ}=\mathcal{L}$ we finish the proof. \\ \\
\begin{Remark}
In case $n \geq 3$, we could use a shorter proof to see that $\Phi$ is linear. Notice that $\Phi$ sends cones to cones, and preserves intersections and convex-hulls of unions of cones. This is shown easily by using properties of $\sup$ and $\hat{\inf}$ from lemma \ref{InfSupLemma}: Define functions which are zero on the cone and $\infty$ everywhere else. The intersection is given by $\sup$ of the functions, and the convex hull is given by $\hat{\inf}$. Observe that the functions have values of $0$ and $\infty$ only, the inequalities in lemma \ref{InfSupLemma} become equalities, and the property holds.
Using Schneider's theorem \cite{S2008}, we conclude that $\Phi$ is linear. This means that $\Phi(x) = Bx$ for some $B \in GL(n)$. \\ \\

\end{Remark}

\section{Stability On the Class of Non-Negative Convex Functions}
We now proceed to the proof of theorem (\ref{GeneralCvxStability}). Again, like in the previous case, we will need a family of extremal functions and some properties of their behaviour under our transform. The extremal family of function we will use in the case are what we call here "delta" functions $D_\theta + c$, mentioned before.
\subsection{Preservation of $\hat{\sup}$ and $\hat{\inf}$}
Clearly, properties (\ref{sup_stab}) and (\ref{inf_stab}) hold in this case too, and the proof of lemma \ref{InfSupLemma} can be applied verbatim.

\subsection{Behaviour of "Delta" Functions}
\subsubsection{Delta Functions are Mapped to Delta Functions}
We will show that $T$ maps the class of ``delta'' functions $\{D_\theta + c\}$ to itself and does so bijectively . Assume $T(D_\theta + c) = f$. We want to show that the support of $f$ has exactly one point. Assume there exist two functions $g$ and $h$ such that $g \geq f$ and $h \geq f$. Due to surjectivity we may write: $g = \tilde{c}T\varphi$ and $h = \tilde{c}T\psi$. Hence,
\begin{subequations}
\begin{align}
T(D_\theta + c) &\leq \tilde{c}T\varphi \\
T(D_\theta + c) &\leq \tilde{c}T\psi.
\end{align}
\end{subequations}
Condition (\ref{TRevEquiv}b) now implies that $D_\theta + c \leq \varphi$ and $D_\theta + c  \leq \psi$. This means that both $\varphi$ and $\psi$, are of the form $D_\theta + \alpha_i$. Thus, they are comparable, and without loss of generality we may assume that $\varphi > \psi$.  Applying condition (\ref{TDefinition}a), we get that $h \leq \tilde{C}g$. But, if the support of $f$ has two or more points, we can easily find two functions greater than $f$, but not comparable up to $\tilde{C}$. So we conclude that $f$ is supported at one point only, and has the form $D_\theta + c'$.

\subsubsection{Only Delta Functions are Mapped to Delta Functions}
Now assume that $Tf = D_\theta + c$ and that the support of $f$ has at least two points $x_0$ and $x_1$, with values $c_0$ and $c_1$. Then $D_{x_0} + c_0 \geq f$ and $D_{x_1} + c_1 \geq f$. Applying condition (\ref{TDefinition}a), we get $\tilde{C}T(D_{x_i} + c_i) \geq D_\theta + c$.  According to the previous lemma $T(D_{x_i} + c_i) = D_{y_i} + a_i$, but they must be comparable (since they are greater than $D_\theta + c$), so $y_1 = y_2=\theta$. But this also implies that the sources are comparable, up to a constant $\tilde{C}$, hence $x_1 = x_2$.

\subsubsection{Delta Functions are Mapped in Fibres}
Since $D_\theta + c > D_\theta$, we get that $T(D_\theta) \leq \tilde{C}T(D_\theta + c) = D_{\theta'} + c'$. We know that $T(D_\theta) = D_\varphi + \alpha$. So $D_\varphi + \alpha \leq D_{\theta'} + c'$, which means that $\varphi = \alpha$ and $T(D_\theta + c) = T(D_\theta) + c''$. We see that all the delta functions on the fibre $x=\theta$ are mapped to delta functions on the fiber $T(D_\theta)$.

\subsection{The Mapping Rule for $D_\theta+c$}
Assume that the delta functions are mapped by the rule $T(D_\theta + c)=D_{\phi(\theta)} + \psi(\theta, c)$. Notice that due to  the property of mapping in fibres, $\phi$ does not depend on $c$.
Now we analyze the behavior of $\phi(\theta)$.  Take $D_{\theta_0}+c_0$ and $D_{\theta_1}+c_1$, and define $g=\hat{\inf}(D_{\theta_0}+c_0, D_{\theta_1}+c_1)$. Consider $\theta=\lambda \theta_0 + (1-\lambda) \theta_1$, and $c=\lambda c_0 + (1-\lambda) c_1$. Obviously $g \leq D_{\theta}+c$, and this is not true for any $D_{\theta}+c'$ where $c' < c$. Apply property (\ref{TDefinition}a) and use (\ref{inf_stab}) to write:
\[
\tilde{c}^2\hat{\inf}(D_{\phi(\theta_0)}+\psi(\theta_0, c_0), D_{\phi(\theta_1)}+\psi(\theta_1, c_1)) \leq T(g) \leq \tilde{C}T(D_\theta+c)=D_{\phi(\theta)}+\tilde{C}\psi(\theta, c).
\]
This inequality implies that $\phi(\theta)$ is on the interval $[\phi(\theta_0), \phi(\theta_1)]$, since otherwise $T(g)$ and $\tilde{C}T(D_\theta+c)$ would not be comparable. So $\phi:\R^n\rightarrow \R^n$ sends intervals to intervals, hence it must be affine (see \cite{AM5}). Thus, there exists $A \in GL(n)$ and $b \in R^n$ such that $\phi(x)=Ax+b$. \\

We know that according to properties (\ref{TDefinition}a) and (\ref{TDefinition}b), for a given $\theta$, $\psi$ satisfies (\ref{temp1}) and (\ref{temp2}). Applying lemma \ref{almost_monotonic_lemma} we find a monotone function $\omega_{\theta}$ that satisfies the following: \[\tilde{c}\omega_{\theta}(t) \leq \psi(\theta, t) \leq \omega_\theta(t).\]
If $c' < c$,  then $g \nleq D_{\theta}+c$, and $Tg \nleq \tilde{c}(D_{\phi(\theta)}+\psi(c'))$. Hence by property (\ref{inf_stab}) $\tilde{C}\hat{\inf}(D_{\phi(\theta_0)}+\psi(\theta_0, c_0), D_{\phi(\theta_1)}+\psi(\theta_1, 2c_1)) \nleq \tilde{c}^2(D_{\phi(\theta)}+\omega_\theta(c'))$. Since $\omega_\theta$ is monotone and the last statement applies for all $c'<c$, we can conclude that
\[
\psi(\theta, c) \leq \omega_\theta(c) \leq \tilde{C}^3(\lambda(\psi(\theta_0, c_0) + (1-\lambda)\psi(\theta_1, c_1)).
\]
But, on the other hand, we have
\[
\psi(\theta, c) \geq \tilde{c}^3(\lambda(\psi(\theta_0, c_0) + (1-\lambda)\psi(\theta_1, c_1)).
\]
So, we know that for every $(x,c)$ and $(y,d)$ in $\R^n \times\R^+$, $\psi$ satisfies the following:
\begin{equation}
c(\lambda \psi(x,c) + (1-\lambda) \psi(y,d)) \leq \psi(\lambda (x,c) + (1-\lambda)(y,d)) \leq C(\lambda \psi(x,c) + (1-\lambda) \psi(y,d)),
\label{quasi_linearity}
\end{equation}
and $\psi(x,0) = 0$.
\subsection{Proving Stability}
According to lemma \ref{Triangle_Height_Function_Lemma}, we know that there exists a constant such that $\tilde{c}\beta d \leq \psi(\theta, d) \leq \tilde{C}\beta d$. Recall also that there exists $A \in GL(n)$ and a vector $b$, such that $T(D_\theta)=D_{A\theta+b}$. This means that $D_{A\theta+b} + \tilde{c}\beta d \leq T(D_\theta + d) \leq D_{A\theta+b} + \tilde{C}\beta d$. \\
We know that any function $f(x) \in Cvx^+(\R^n)$ can be described by "Delta" functions: $f(x) = (\hat{\inf}(D_y + f(y)))(x)$. So,
\begin{eqnarray*}
(Tf)(x) &=& T(\hat{\inf_y}(D_y + f(y)))(x)  \leq \tilde{C}(\hat{\inf_y}T(D_y+f(y)))(x) \\ \\
&\leq& \tilde{C}(\hat{\inf_y}D_{Ay+b} + C\beta f(y))(x) = \tilde{C}(\beta f(A^{-1}(x-b))
\end{eqnarray*}
The lower bound is obtained in the same way, so we come to:
\[
\tilde{c}\beta f(A^{-1}(x-b)) \leq (Tf)(x) \leq \tilde{C}\beta f(A^{-1}(x-b)),
\]
as required. \\ \\
The authors would like to express their sincere appreciation to Prof. Vitali Milman and Prof. Shiri Artstein-Avidan for their support, advice and discussions.

\end{document}